%% file: main.tex
\documentclass[11pt,dvipsnames]{amsart}
\input{preamble}
\usepackage[margin=1in]{geometry}



\title{The $K$-theory of finite Tambara fields: \\ away from $p$}
\date{}

\author[N. Wisdom]{Noah Wisdom}\thanks{Department of Mathematics, Stockholm University.
\texttt{noah.wisdom@math.su.se}}

\begin{document}

\begin{abstract}
    In previous work, the author and David Chan computed the algebraic $K$-theory of the constant $C_2$-Tambara field with value the field with two elements, using a method which fails at odd primes. Herein we make progress towards the corresponding odd primary computations using a completely new idea. Particularly, we show that the $K$-theory groups of any constant $C_{p^n}$-Tambara field with value a characteristic $p$ finite field are torsion, and we completely determine these groups after inverting $p$. The away-from-$p$-torsion satisfies a simple pattern predicted by previous work, and a computer-aided computation shows that the $p$-power torsion is nontrivial in general.
\end{abstract}

\maketitle

\input{Introduction}
\input{MainResults}

\bibliographystyle{alpha}
\bibliography{ref}

\end{document}

%% file: preamble.tex
\usepackage{xy}
\usepackage{graphicx, amsmath, amssymb, amsthm,amsfonts}
\usepackage[mathscr]{eucal}
\usepackage[pagebackref, colorlinks, citecolor=Red, linkcolor=NavyBlue, urlcolor=NavyBlue]{hyperref}
\hypersetup{linktoc=all}
\usepackage{comment}
\usepackage{hyperref}
\usepackage[capitalise, nameinlink]{cleveref}
\usepackage{microtype}
\usepackage{tikz-cd}
\usepackage{spectralsequences}
\usepackage{listings}

\newcommand{\Z}{\mathbb{Z}}

\newcommand{\F}{\mathbb{F}}

\crefname{lemma}{Lemma}{Lemmas}
\crefname{theorem}{Theorem}{Theorems}
\crefname{definition}{Definition}{Definitions}
\crefname{proposition}{Proposition}{Propositions}
\crefname{remark}{Remark}{Remarks}
\crefname{corollary}{Corollary}{Corollaries}
\crefname{equation}{Equation}{Equations}
\crefname{construction}{Construction}{Constructions}
\crefname{ex}{Example}{Examples}
\crefname{example}{Example}{Example}
\crefname{appsec}{Appendix}{Appendices}
\crefname{subsection}{Subsection}{Subsections}

\newtheorem{theorem}{Theorem}[section]
\newtheorem{letterthm}{Theorem}

\newtheorem{lemma}[theorem]{Lemma}

\newtheorem{proposition}[theorem]{Proposition}

\newtheoremstyle{BoldRemark} % name
{10pt}                    % Space above
{10pt}                    % Space below
{\upshape}                   % Body font
{}                           % Indent amount
{\bfseries}                  % Theorem head font
{.}                          % Punctuation after theorem head
{.5em}                       % Space after theorem head
{}  % Theorem head spec (can be left empty, meaning ‘normal’)
\theoremstyle{BoldRemark}
\newtheorem{remark}[theorem]{Remark}

\newtheorem{definition}[theorem]{Definition}

\newtheorem{construction}[theorem]{Construction}

\newtheorem{warning}[theorem]{Warning}

\usepackage[textwidth=25mm, textsize=tiny]{todonotes}

\newcommand{\res}{\mathrm{res}}

\newcommand{\tr}{\mathrm{tr}}

\newcommand{\ev}{\mathrm{ev}}

\newcommand{\idle}{\mathrm{idle}}

%% file: Introduction.tex
\section{Introduction}

Algebraic $K$-theory is a machine which inputs things like categories of modules over a ring, categories with notions of short exact sequences, and stable $\infty$-categories, and which outputs a commutative ring spectrum. In foundational work, Quillen computed the algebraic $K$-theory groups of finite fields \cite{Qui72b}. This result forms one of the most important inputs to trace methods, our most powerful tool for accessing the arithmetic information contained in $K$-theory groups. For example, Antieau--Krause--Nikolaus \cite{AKN24}, building on seminal work of Dundas--Goodwillie--McCarthy \cite{DGM13} and Nikolaus--Scholze \cite{NS18} and B\"{o}kstedt--Hsiang--Madsen \cite{BHM93} and many others, used Quillen's computation to access the $K$-theory groups of the rings $\Z/p^n$. This put to rest a decades-old open problem. 

We now introduce our second major player. Tambara functors generalize the notion of ring to equivariant algebra. Really they generalize the notion of ``ring with $G$-action," and they are the algebraic shadows of equivariant ring spectra \cite[Section 7.2]{Bru07}. These look like assignments $R : G/H \mapsto R(G/H)$ of a ring $R(G/H)$ to each finite transitive $G$-set $G/H$, along with numerous internal structure maps encoding change-of-group operations. For example, $G/H \mapsto RO(H)$ defines the real representation Tambara functor and it encodes $\oplus$- and $\otimes$-induction of representations. In keeping with the program of algebraic topology, many homotopy theorists have studied Tambara functors in the past decade following important work of Hill--Hopkins--Ravenel \cite{HHR}. 

Following Nakaoka, a Tambara functor is said to be \emph{field-like} if it has no nontrivial ideals \cite{Nak11a}; we also call such things Tambara fields. If $G$ acts on an ordinary field $\F$, then the assignment $G/H \mapsto \F^H$ defines a \emph{fixed-point} Tambara functor, and it is field-like. Previous work of the author completely classifies field-like Tambara functors when $G = C_{p^n}$ is cyclic of order $p^n$ \cite{Wis24}.

At this juncture, one may straightforwardly ask two questions. First: what are the algebraic $K$-theory groups of a given Tambara functor? Precisely, there is a notion of module over (the Green functor underlying a) Tambara functor, and we can try to compute the $K$-theory of the exact category of finitely generated projective modules. Second: what arithmetic information is contained in the $K$-theory groups of a Tambara functor?

We make significant progress towards the first question levelwise-finite $C_{p^n}$-Tambara fields, using an entirely new method unique to equivariant algebra. The most basic nontrivial case is that of the \emph{constant} Tambara functors $\underline{\F}$ with value a finite characteristic $p$ field $\F$, given by $G/H \mapsto \F$. Their modules are also known as \emph{cohomological} Mackey functors (over $\F$), and Yoshida observed that these are the same thing as modules over an appropriate Hecke algebra \cite{Yos83}.

\begin{letterthm}\label{letterthm:A}
    Let $\F$ be a finite field of characteristic $p$ and $G = C_{p^n}$. The groups $K_i(\underline{\F})$ are finite when $i \geq 1$, and we have  
    \[ 
        K_i(\underline{\F})[p^{-1}] \cong \begin{cases}
            \Z[p^{-1}]^{\oplus n+1} & i=0 \\
            (\mathbb{Z}/(q^j-1))^{\oplus n+1} & i = 2j-1 \\
            0 & \mathrm{else} . 
        \end{cases}
    \]
\end{letterthm}

We emphasize that the $p = 2$, $n=1$ case of this may be recovered from the computation $K(\underline{\F_2}) \cong K(\F_2)^{\oplus 2}$ of Chan and the author \cite{CW25}. We will see later that $K_1(\underline{\F_3})$ has $3$-torsion when $G = C_3$, so new ideas will be needed to get a full computation when $p > 2$ or $n > 1$. Furthermore, we note that the methods herein are completely distinct from the methods of \cite{CW25}, so that we have two completely different theoretical approaches giving the same computational result (when $p=2$ and $n=1$).

Finally, we are also able to compute $K_0$ for many categories of cohomological $C_{p^n}$-Mackey functors, generalizing the case wherein $Q$ is a field determined by Chan and the author \cite{CW25}. The argument in \cite{CW25} is a few pages of intricate linear algebra, whereas we derive our result as an immediate consequence of our main technical lemma which, when reduced to a slogan, says that $\underline{Q}$ is a nilpotent extension of $Q^{\times n+1}$.

\begin{letterthm}\label{letterthm:B}
    Let $Q$ be an $\F_p$-algebra and $G = C_{p^n}$. We have an isomorphism
    \[ 
        K_0(\underline{Q}) \cong K_0(Q)^{\oplus n+1} 
    \]
    natural in $Q$.
\end{letterthm}

\subsection{Notation and conventions}

Throughout, $Q$ will always denote a ring in which $p = 0$, and $\F$ is a finite field of characteristic $p$. A reader interested only in \cref{letterthm:A} may safely substitute $Q = \F$ or $Q = \F_p$ everywhere. A module over a Tambara functor is just a module over the underlying Green functor (we also briefly recall the definition explicitly). If $M$ is a Mackey or Tambara functor, we implicitly extend $M$ via the formula $M(X \sqcup Y) = M(X) \times M(Y)$.

\subsection{Acknowledgments}

The author thanks Howard Beck, David Chan, Arun Debray, Bert Guillou, David Mehrle, Marc Gotliboym, Nat Stapleton, and Chuck Weibel for helpful conversations related to the content of this article. Additionally, the author thanks Ben Antieau and Noah Riggenbach for graciously explaining many minor technical aspects of algebraic $K$-theory to the author, and Mike Hill for pointing out that (and how precisely) the idle condition is stronger than the Weyl-invariant condition for non-abelian groups. The author regrets that many of these acknowledgments concern conversations over material that was ultimately revised out of the article while it mutated through the many approaches that turned out to be dead ends.

The author gratefully acknowledges support of NSF grant DMS-2602147.

\subsection{AI Declaration} 

No AI was used for this research.

%% file: MainResults.tex
\section{Algebraic preliminaries: variations on the Mackey algebra}

From here on out, we will fix a prime $p$. The ambient group will be $G = C_{p^n}$ for some $n \geq 1$, so $\underline{Q}$ denotes the constant $C_{p^n}$-Tambara functor with value a fixed $\F_p$-algebra $Q$. The notion of module over (the Green functor underlying a) Tambara functor is reviewed in \cite{CW25}; here is a brief definition.

A $\underline{Q}$-module $M$ is a collection of $Q$-modules $M(G/H)$ for each $H \subset G$, along with the data of:
\begin{enumerate}
    \item a \emph{restriction} $\res_H^L : M(G/L) \rightarrow M(G/H)$ whenever $H \subset L$,
    \item a \emph{transfer} $\tr_H^L : M(G/H) \rightarrow M(G/L)$ whenever $L \subset H$, and 
    \item an action of the group $G/H$ on $M(G/H)$, with $g \in G$ acting by the \emph{conjugation} $c_g$;
\end{enumerate}
satisfying: 
\begin{enumerate}
    \item restrictions and transfers commute with conjugations: $\res_H^L c_g = c_g \res_H^L$ and $\tr_H^L c_g = c_g \tr_H^L$, 
    \item the double coset formula: $\res_H^L \tr_H^L = \sum_{g \in L/H} c_g$, and 
    \item $M$ is cohomological: $\tr_H^L \res_H^L = [L:H] = 0$ (since $p = 0$ in $Q$).
\end{enumerate}
A morphism of $\underline{Q}$-modules $M \rightarrow M'$ is a collection of $Q$-module maps $M(G/H) \rightarrow M'(G/H)$ which commute with all structure maps.

\begin{definition}\label{def:idleness}
    We say a $\underline{Q}$-module $M$ is \emph{idle} if, for each $H \subset G$, the Weyl group $W_G H \cong G/H$ of $H$ acts trivially on $M(G/H)$. We will use $\underline{Q} \text{-} \idle$ to denote the full subcategory of such.
\end{definition}

\begin{warning}
    The reader is cautioned that the naive generalization of the notion of idleness to other groups does not seem correct, at least from the $K$-theoretic perspective. Indeed, one may classify, for example, Nullstellensatzian objects in the category of Weyl-invariant Tambara functors (and the answer is that they are precisely those $G$-Tambara functors $\underline{\F}$ for $\F$ an algebraically closed field). However, there seems to be a slightly stronger notion which is rooted more firmly in foundational aspects of equivariant algebra, and which interacts more nicely with the idle algebra $\nu$. 
    
    For example, as Hill has pointed out to the author, one may consider product-preserving functors out of the Lindner (span) category in the free finite coproduct completion of the category of subgroups of $G$, and take this as a definition of ``idle".
\end{warning}

For each finite group $G$, Th\'{e}venaz and Webb construct an associative ring, the \emph{Mackey algebra} $\mu$ of $G$, whose category of modules is equivalent to the category of Mackey functors \cite[Section 3]{TW95}.

\begin{definition}[{cf. \cite[Section 3]{TW95}}]
    The Mackey algebra $\mu$ is the free associative algebra generated by elements 
    \begin{enumerate}
        \item \emph{transfers} $T_K^H$ for all subgroup inclusions $K \subset H$,
        \item \emph{restrictions} $R_K^H$ for all subgroup inclusions $K \subset H$,
        \item \emph{conjugations} $c_g$ for all $g \in G$,
    \end{enumerate}
    subject to the relations
    \begin{enumerate}
        \item restrictions compose as $R_J^K R_K^H = R_J^H$ whenever $J \subset K \subset H$,
        \item transfers compose as $T_K^H T_J^K = T_J^H$ whenever $J \subset K \subset H$,
        \item conjugations compose as $c_g c_h = c_{gh}$ whenever $g, h \in G$,
        \item restrictions and conjugations compose as $R_{gKg^{-1}}^{gHg^{-1}} c_g = c_g R_K^H$ whenever $K \subset H$ and $g \in G$,
        \item transfers and conjugations compose as $T_{gKg^{-1}}^{gHg^{-1}} c_g = c_g T_K^H$ whenever $K \subset H$ and $g \in G$,
        \item the double coset formula 
                \[ 
                    R_J^H T_K^H = \sum_{x \in J \backslash H / K} T_{J \cap xKx^{-1}}^J c_x R_{xJx^{-1} \cap K}^K
                \]
                holds whenever $J \subset H$ and $K \subset H$, 
        \item trivial restrictions agree with trivial transfers: $R_H^H = T_H^H$ for all $H \subset G$, and
        \item all other as-yet unspecified products of two symbols vanish (e.g. $R_H^{H'} R_L^{L'} = 0$ whenever $L \neq H'$).
    \end{enumerate}
    Note that our $T_K^H$ is Th\'{e}venaz--Webb's $I_K^H$, and Th\'{e}venaz--Webb's $c_{g,H}$ is equal to $c_g R_H^H$. Additionally, assuming relations $(1)$ and $(2)$ above, Th\'{e}venaz--Webb's relation $(0')$ is equivalent to our relation $(7)$.
\end{definition}

The key technical tool for us is an idle variation of this construction.

\begin{construction}\label{cons:idle-algebra}
    Define a finite $G$-set 
    \[ 
        X := \sqcup_{H \subset G} G/H 
    \]
    and consider the functor 
    \[ 
        \ev_X : M \mapsto M(X) \cong \prod_H M(G/H) . 
    \]
    As a functor from Mackey functors to abelian groups, $\ev_X$ is represented by a \emph{free} Mackey functor $F_X$ for formal reasons. Then $\mu$ is recovered as the endomorphism ring of $F_X$. In particular, $\mu$ is given additively as the group $F_X(X)$. Crucially, the equivalence between $\mu$-modules and Mackey functors is given in one direction by $\ev_X$.

    We define a Mackey algebra $\mu_Q$ for $\underline{Q}$-modules, an associative ring whose category of modules is equivalent to the category of $\underline{Q}$-modules. Recall that the box product $\boxtimes$ of Mackey functors is defined using the Day convolution with respect to the tensor product of abelian groups and the coproduct in the Lindner category $\mathrm{Span}(G \text{-} \mathrm{Set}^{\mathrm{fin}})$. Additively, we define
    \[ 
        \mu_{Q} := \ev_X(F_X \boxtimes \underline{Q}) . 
    \]
    
    Let $\mathcal{A}_G$ denote the Burnside Tambara functor. 
    When $Q = \F_p$, the unit $\mathcal{A}_G \rightarrow \underline{\F_p}$ is surjective, hence
    \[ 
        \ev_X(F_X \cong F_X \boxtimes \mathcal{A}_G \rightarrow F_X \boxtimes \underline{\F_p})
    \]
    is a surjection of abelian groups $\mu \rightarrow \mu_{\F_p}$. The kernel of this surjection is precisely the two-sided ideal generated by the relations $p = 0$ and $T_H^L R_H^L = [L:H]$ in $\mu$. Thus $\mu_{\F_p}$ obtains a ring structure. We equip $\mu_{Q}$ with the ring structure determined by the additive isomorphism $\mu_{Q} \cong \mu_{\F_p} \otimes_{\F_p} Q$.

    Next, we define the idle algebra $\nu_{Q}$, whose modules are precisely \emph{idle} $\underline{Q}$-modules. Recall the conjugation elements $c_g$ in the Th\'{e}venaz--Webb Mackey algebra $\mu$, which determine the Weyl actions in a Mackey functor. We will use $c_g$ to denote also their image in $\mu_{Q}$. The idle algebra is defined by 
    \[ 
        \nu_{Q} := \mu_{Q}/(c_g - 1 | g \in G) , 
    \]
    the quotient of $\mu_{Q}$ by the two-sided ideal generated by the (central) elements $c_g-1$.
    It remains to identify $\nu_{Q}$-modules with idle $\underline{Q}$-modules.

    If $M$ is an idle $\underline{Q}$-module, then the $c_g$ act trivially, so the module structure map 
    \[ 
        \mu_{Q} \rightarrow \mathrm{End}_{\mathrm{Ab}}(\ev_X(M)) 
    \] 
    factors uniquely through $\nu_{Q}$. Thus every idle $\underline{Q}$-module may be regarded in a unique way as an $\nu_{Q}$-module. Conversely, given any $\nu_{Q}$-module $N$ classified by a ring map 
    \[ 
        \nu_{Q} \rightarrow \mathrm{End}_{\mathrm{Ab}}(N)
    \] 
    precomposition with the quotient $\mu_{Q} \rightarrow \nu_{Q}$ defines a $\underline{Q}$-module $N'$. The Weyl group actions arise from the scalar action of the $c_g$ on $N$, which are trivial by construction. Thus $N'$ is idle, as desired. These constructions may be seen to be inverse to each other, so the claim follows.
\end{construction}

Let $Q$ be finite. Since $\mu_Q$ is a finite free $Q$-module, the groups $K_i(\underline{Q})$ are finite when $i \geq 1$ (Weibel provides a short proof of this result in \cite[Proposition IV.1.16]{Wei13} and attributes it to Kuku). Next, it is crucial to know that the kernel of $\mu_Q \rightarrow \nu_Q$ is nilpotent.

\begin{lemma}\label{lem:idlization-kernel-is-nilpotent}
    Let 
    \[ 
        I = \mathrm{Ker}(\mu_{Q} \rightarrow \nu_{Q}) . 
    \]
    Then $I$ is a nilpotent ideal.
\end{lemma}

\begin{proof}
    Since $G$ is abelian, the presentation of $\mu$ in \cite{TW95} implies that each $c_g$ is in the center of the Mackey algebra $\mu$. It follows that $c_g$ is central in $\mu_{\F_p}$, hence in $\mu_{Q}$. Recall that $p = 0$ in $Q$. Since 
    \[ 
        (c_g - 1)^{p^n} = c_{g^{p^n}} - 1 = c_e - 1 = 0 
    \] 
    $I$ is generated by finitely many central nilpotent elements. Consequently $I$ is nilpotent.
\end{proof}

With these algebraic considerations out of the way, we can now compute some $K$-theory.

\section{Reduction to idle \texorpdfstring{$K$}{K}-theory}

\begin{definition}
    The \emph{idle} $K$-theory of $\underline{Q}$ is defined as 
    \[ 
        K^\idle(\underline{Q}) := K(\nu_{Q}) \cong K(\mathrm{Proj}(\underline{Q} \text{-} \idle^{\mathrm{f.g.}}) )
    \]
\end{definition}

\begin{proposition}\label{prop:LES-for-relative}
    There is a fiber sequence of $K$-theory spectra 
    \[ 
        K(\mu_{Q},I) \rightarrow K(\underline{Q}) \rightarrow K^\idle(\underline{Q}) . 
    \]
\end{proposition}

\begin{proof}
    Unwinding definitions, this is the fiber sequence of $K$-theory spectra associated to the ring surjection $\mu_{Q} \rightarrow \nu_{Q}$ (cf. \cite[IV.1.11]{Wei13}). In particular, $K(\mu_Q,I)$ is \emph{defined} as the fiber above.
\end{proof}

Next we recall an exercise in Weibel's textbook \cite{Wei13} and a result recorded by Bass \cite{Bas68}.

\begin{theorem}\label{thm:Weibel and Bass}
    Let $I$ be a nilpotent two-sided ideal of an associative ring $S$ such that $p^a I = 0$ for some $a$. Then the relative $K$-theory groups $K_*(S,I)$ are $p$-groups and $K_0(S) \rightarrow K_0(S/I)$ is an isomorphism.
\end{theorem}

\begin{proof}
    The first claim is \cite[Exercise IV.1.18]{Wei13}. The second claim on $K_0$ may be found as \cite[Proposition 1.3 in Chapter IX]{Bas68}.
\end{proof}

Applying \cref{thm:Weibel and Bass} to $\mu_Q \rightarrow \nu_Q$ yields the following.

\begin{theorem}\label{thm:rel-K-groups-are-p-torsion;inverting-p-is-free-and-easy-and-leads-to-a-happier-paradise}
    The relative $K$-theory groups 
    \[ 
        K_*(\mu_{Q},I) 
    \] 
    are $p$-groups. Consequently 
    \[ 
        K(\underline{Q})[p^{-1}] \rightarrow K^\idle(\underline{Q})[p^{-1}] 
    \] 
    is an equivalence of commutative ring spectra. Moreover 
    \[ 
        K_0(\underline{Q}) \rightarrow K_0^{\idle}(\underline{Q}) 
    \] 
    is an isomorphism.
\end{theorem}

\begin{remark}
    One can squeeze out a little more generality from our methods so far. Let $R$ be any $C_p$-Tambara functor such that each Weyl group acts trivially and $p = 0$ in $R(C_p/e)$. For instance the quotient of the Burnside Tambara functor by the ideal generated by $p$ in the $C_{p^n}/e$-level has this form; it is also known as $N_e^{C_p} \F_p$ and it satisfies $N_e^{C_p} \F_p(C_p/C_p) \cong \Z/p^2$. Then one can again construct the Mackey and idle algebras $\mu_R$ and $\nu_R$ of $R$. The same argument shows that $K(R) \rightarrow K^\idle(R)$ becomes an equivalence after inverting $p$, and that $K_0(R) \cong K_0^\idle(R)$. In this case one should obtain $K(R)[p^{-1}] \cong K(R(C_p/e) \times R(C_p/C_p))[p^{-1}]$ (or---possibly---one should replace $R(C_p/C_p)$ with the geometric fixed points $R(C_p/C_p)/\tr$ of $R$).
\end{remark}

\section{Computation of idle \texorpdfstring{$K$}{K}-theory}

We begin with an algebraic preliminary which forms the technical heart of this entire paper.

\begin{lemma}\label{lem:reduce-to-product-using-nilpotence}
    There is an associative ring map $\nu_Q \rightarrow Q^{\times n+1}$ given as a finite composition of ring maps $S \rightarrow S'$ with nilpotent kernel.
\end{lemma}

\begin{proof}
    We inductively quotient out by transfer and restriction elements in $\nu_Q$, starting with the biggest transfers and restrictions. The base case illustrates the idea: let $J$ is the two-sided ideal of $\nu_Q$ generated by $R_e^{C_{p^n}}$. One checks using Th\'{e}venaz--Webb presentation of the Mackey algebra, and also the cohomological assumption and the fact that the elements $c_g - 1$ are zero in $\nu_{Q}$, that $J = Q \cdot R_e^{C_{p^n}}$. Since $R_e^{C_{p^n}}$ squares to zero, $J^2 = 0$. Thus $\nu_Q \rightarrow \nu_Q/(R_e^{C_{p^n}})$ is the first map in our composition. A similar situation occurs for the ideal of $\nu_Q/(R_e^{C_{p^n}})$ generated by (the image of) $R_e^{C_{p^n}}$. Note that $Q^{\times n+1}$ is the subalgebra of $\nu_Q$ generated by all the idempotent elements $R_H^H$. 

    In the inductive step we have some quotient ring $S$ of $\nu_Q$ for which the quotient is a composition of quotients with nilpotent kernels. Choose a subgroup pair $H \subset H'$ of maximal index $|H'/H|$ among those pairs for which there is either a nonzero $R_H^{H'}$ or $T_{H'}^H$ in $S$. Of course, there are only finitely many such elements, since there are only finitely many such elements in the Mackey algebra $\mu$. If the maximal index is $1$, then $S$ is the subalgebra $Q^{\times n+1}$ of $\nu_Q$ generated by the elements $R_H^H$, in which case the induction terminates. Otherwise $H \neq H'$, then we have some $R_H^{H'}$ or $T_{H'}^H$ nonzero. 

    Recall that the product $R_H^{H'} R_L^{L'}$ is zero in $\mu_Q$, hence in our quotient, unless $H = L'$, in which case the product is the image of $R_L^{H'}$ in our quotient. Note $L \subset H \subset H'$, so $|H'/H| \geq |H'/L|$, and equality holds if and only if $L = H$. Since the index $|H'/H|$ is maximal in the sense of the previous paragraph (and because $\underline{Q}$-modules always satisfy $\tr_H^{H'} \res_H^{H'} = 0 = \res_H^{H'} \tr_H^{H'}$), the only nonzero product of $R_H^{H'}$ with another algebra generator $R_L^{L'}$ or $T_L^{L'}$ occurs when $L = L'$, in which case $L = H$ or $L' = H'$ and the product is equal to $R_H^{H'}$. Consequently the ideal $J$ generated by $R_H^{H'}$ is just $Q \cdot R_H^{H'}$, which is again nilpotent. The same argument goes through verbatim with $R_H^{H'}$ replaced by $T_H^{H'}$. This produces $S \rightarrow S/J =: S'$, a ring surjection with nilpotent kernel. We return to the inductive step, replacing $S$ with $S'$. This process reduces the $Q$-rank of $S$ by at least one, hence terminates.
\end{proof}

\begin{remark}
    The map $\mu_Q \rightarrow \nu_Q \rightarrow Q^{\times n+1}$ is split by the map sending the primitive idempotents of $Q^{\times n+1}$ to the idempotents $R_H^H$ of $\mu_Q$. Consequently $K(Q)^{\times n+1}$ is a retract of $K(\underline{Q})$ and of $K^\idle(\underline{Q})$.
\end{remark}

Combining \cref{lem:idlization-kernel-is-nilpotent,lem:reduce-to-product-using-nilpotence,thm:Weibel and Bass} we deduce \cref{letterthm:B}. Using Quillen's computation of $K(\F)$ \cite{Qui72b}, we obtain \cref{letterthm:A} from the special case $Q = \F$ of \cref{thm:main}:

\begin{theorem}\label{thm:main}
    Consider map $\nu_Q \rightarrow Q^{\times n+1}$ of \cref{lem:reduce-to-product-using-nilpotence}. The induced map 
    \[ 
        K^\idle(\underline{Q})[p^{-1}] \rightarrow K(Q^{\times n+1})[p^{-1}] \cong K(Q)[p^{-1}]^{\times n+1}
    \] 
    is an equivalence and the induced map 
    \[ 
        K_0^\idle(\underline{Q}) \rightarrow K_0(Q^{\times n+1})
    \] 
    is an isomorphism.
\end{theorem}

\begin{proof}
    Apply \cref{thm:Weibel and Bass} successively to the conclusion of \cref{lem:reduce-to-product-using-nilpotence}.
\end{proof}

Although we do not use it here, we observe that we are in the situation of the Dundas--Goodwillie--McCarthy theorem \cite{DGM13}, hopeful that it may be of future use. In particular, it suggests a way to approach the ``at $p$" part of $K(\underline{\F})$---by computing $TC(\mu_\F)$.

\begin{proposition}\label{prop:DGM}
    If we define $TC(\underline{Q}) := TC(\mu_{Q})$ and $TC^\idle(\underline{Q}) := TC(\nu_{Q})$, then there is a commutative diagram
    \[ \begin{tikzcd}
        K(\underline{Q}) \arrow[r] \arrow[d] & K^\idle(\underline{Q}) \arrow[d] \arrow[r] & 
        K(Q^{\times n+1}) \arrow[d] \\
        TC(\underline{Q}) \arrow[r] & TC^\idle(\underline{Q}) \arrow[r] & TC(Q^{\times n+1})
    \end{tikzcd} \]
    of spectra in which all squares are pullbacks.
\end{proposition}

Note that our $TC$ is just $TC$ of an associative ring. By contrast, recent work of Chan--Gerhardt--Klang \cite{CGK25} and Hilman--Ramzi \cite{HR26} establish candidates for $G$-equivariant $THH$ which are hoped to lead to a useful $G$-spectrum definition of $TC$ (when the input is a ring $G$-spectrum or Tambara functor); see also forthcoming thesis work of Gotliboym introducing equivariant $TC$. We note moreover that the $C_p$-equivariant topological Hochschild homology groups (in the sense of Chan--Gerhardt--Klang) of $\underline{\F_p}$ have been computed by \cite{AGHKK22} when $p = 2$ and \cite{CGKW26} when $p$ is odd.

\section{\texorpdfstring{An explicit computation of $K_1$}{An explicit computation of K\_1}}

Consider $\underline{\F_3}$ when $G = C_3$. Recently Hofmann developed a part of the OSCAR package for Julia \cite{OSCAR,Nemo_Hecke} which computes $K_1$ of finite rings \cite{Hof25}.
Using the input
\begin{lstlisting}
julia> R, = finite_ring([3, 3, 3, 3, 3, 3], [ZZ[1 0 0 0 0 0; 0 0 0 0 0 0; 0 0 1 0 0 0; 0 0 0 0 0 0; 0 -1 0 0 1 0; 0 -1 0 0 0 1], ZZ[0 0 0 0 0 0; 0 1 0 0 0 0; 0 0 0 0 0 0; 0 0 0 1 0 0; 0 1 0 0 0 0; 0 1 0 0 0 0], ZZ[0 0 0 0 0 0; 0 0 1 0 0 0; 0 0 0 0 0 0; 0 0 0 0 0 0; 0 0 1 0 0 0; 0 0 1 0 0 0], ZZ[0 0 0 1 0 0; 0 0 0 0 0 0; 0 0 0 0 0 0; 0 0 0 0 0 0; 0 0 0 1 0 0; 0 0 0 1 0 0], ZZ[0 -1 0 0 1 0; 0 1 0 0 0 0; 0 0 1 0 0 0; 0 0 0 1 0 0; 0 0 0 0 0 1; 1 1 0 0 0 0], ZZ[0 -1 0 0 0 1; 0 1 0 0 0 0; 0 0 1 0 0 0; 0 0 0 1 0 0; 1 1 0 0 0 0; 0 0 0 0 1 0]]);
julia> Oscar.k1(R)
\end{lstlisting}
for $K_1(\underline{\F_3})$ and 
\begin{lstlisting}
julia> R, = finite_ring([3, 3, 3, 3], [ZZ[1 0 0 0; 0 0 0 0; 0 0 1 0; 0 0 0 0], ZZ[0 0 0 0; 0 1 0 0; 0 0 0 0; 0 0 0 1], ZZ[0 0 0 0; 0 0 1 0; 0 0 0 0; 0 0 0 0], ZZ[0 0 0 1; 0 0 0 0; 0 0 0 0; 0 0 0 0]]);
julia> Oscar.k1(R)
\end{lstlisting}
for $K_1^\idle(\underline{\F_3})$, we compute 
\begin{align*}
    K_1(\underline{\F_3}) & \cong (\Z/2)^2 \times \Z/3 \\
    K_1^\idle(\underline{\F_3}) & \cong (\Z/2)^2 . 
\end{align*}

From this we draw two observations. First, this is consistent with all of our results. Since the input data is merely an explicit description of the idle and Mackey algebras, which have dimension $4$ and $6$ over $\F_3$ respectively, we are somewhat reassured that there is no catastrophic error in our theoretical methods. Second, we obtain a new result: $K_1(\mu_{\F_3},I) = \Z/3$ is not zero in general (and $K_1(\nu_{\F_3} \rightarrow \F_3^2) = 0$), contrary to the expectations of the authors of \cite{CW25}. 

Lastly, motivated by the computation of $K_1^\idle(\underline{\F_3})$, we conjecture $K^\idle(\underline{Q}) \cong K(Q)^{\times n+1}$.

%% file: ref.bib
@inproceedings{Nemo_Hecke,
    author = {Fieker, Claus and Hart, William and Hofmann, Tommy and Johansson, Fredrik},
    title = {Nemo/{H}ecke: Computer Algebra and Number Theory Packages for the {J}ulia Programming Language},
    year = {2017},
    isbn = {9781450350648},
    publisher = {Association for Computing Machinery},
    address = {New York, NY, USA},
    url = {https://doi.org/10.1145/3087604.3087611},
    doi = {10.1145/3087604.3087611},
    booktitle = {Proceedings of the 2017 ACM International Symposium on Symbolic and Algebraic Computation},
    pages = {157–164},
    numpages = {8},
    location = {Kaiserslautern, Germany},
    series = {ISSAC '17}
}

@article {AGHKK22,
    AUTHOR = {Adamyk, Katharine and Gerhardt, Teena and Hess, Kathryn and Klang, Inbar and Kong, Hana Jia},
     TITLE = {Computational tools for twisted topological {H}ochschild
              homology of equivariant spectra},
   JOURNAL = {Topology Appl.},
  FJOURNAL = {Topology and its Applications},
    VOLUME = {316},
      YEAR = {2022},
     PAGES = {Paper No. 108102, 26},
      ISSN = {0166-8641,1879-3207},
   MRCLASS = {55P91 (16E40 19D55 55T99)},
  MRNUMBER = {4438947},
MRREVIEWER = {Bj\o rn\ Ian\ Dundas},
       DOI = {10.1016/j.topol.2022.108102},
       URL = {https://doi.org/10.1016/j.topol.2022.108102},
}

@misc{AKN24,
      title={On the {$K$}-theory of $\mathbf{Z}/p^n$}, 
      author={Benjamin Antieau and Achim Krause and Thomas Nikolaus},
      year={2024},
      eprint={2405.04329},
      archivePrefix={arXiv},
      primaryClass={math.KT},
      url={https://arxiv.org/abs/2405.04329}, 
}

@book{Bas68,
  title={Algebraic K-theory},
  author={Bass, H.},
  series={Hyman Bass},
  year={1968},
  publisher={W. A. Benjamin}
}

@article {Bru07,
    AUTHOR = {Brun, Morten},
     TITLE = {Witt vectors and equivariant ring spectra applied to
              cobordism},
   JOURNAL = {Proc. Lond. Math. Soc. (3)},
  FJOURNAL = {Proceedings of the London Mathematical Society. Third Series},
    VOLUME = {94},
      YEAR = {2007},
    NUMBER = {2},
     PAGES = {351--385},
      ISSN = {0024-6115,1460-244X},
   MRCLASS = {55P43 (13K05 55N22)},
  MRNUMBER = {2308231},
MRREVIEWER = {Andrey\ Yu.\ Lazarev},
       DOI = {10.1112/plms/pdl010},
       URL = {https://doi.org/10.1112/plms/pdl010},
}

@article{BHM93,
 author = {B{\"o}kstedt, M. and Hsiang, W. C. and Madsen, I.},
 title = {The cyclotomic trace and algebraic {K}-theory of spaces},
 fjournal = {Inventiones Mathematicae},
 journal = {Invent. Math.},
 issn = {0020-9910},
 volume = {111},
 number = {3},
 pages = {465--539},
 year = {1993},
 language = {English},
 doi = {10.1007/BF01231296},
 url = {https://eudml.org/doc/144086},
 zbMATH = {203552},
 Zbl = {0804.55004}
}

@misc{CGK25,
      title={Trace methods for equivariant algebraic {K}-theory}, 
      author={David Chan and Teena Gerhardt and Inbar Klang},
      year={2025},
      eprint={2505.11327},
      archivePrefix={arXiv},
      primaryClass={math.AT},
      url={https://arxiv.org/abs/2505.11327}, 
}

@misc{CGKW26,
      title={Computations in Equivariant Topological Hochschild Homology}, 
      author={David Chan and Marc Gotliboym and Inbar Klang and Noah Wisdom},
      year={2026},
      eprint={2608.11376},
      archivePrefix={arXiv},
      primaryClass={math.AT},
      url={https://arxiv.org/abs/2608.11376}, 
}

@misc{CW25,
      title={The algebraic {$K$}-theory of {G}reen functors}, 
      author={David Chan and Noah Wisdom},
      year={2025},
      eprint={2508.14207},
      archivePrefix={arXiv},
      primaryClass={math.KT},
      url={https://arxiv.org/abs/2508.14207}, 
}

@book {DGM13,
    AUTHOR = {Dundas, Bj\o rn Ian and Goodwillie, Thomas G. and McCarthy,
              Randy},
     TITLE = {The local structure of algebraic {K}-theory},
    SERIES = {Algebra and Applications},
    VOLUME = {18},
 PUBLISHER = {Springer-Verlag London, Ltd., London},
      YEAR = {2013},
     PAGES = {xvi+435},
      ISBN = {978-1-4471-4392-5; 978-1-4471-4393-2},
   MRCLASS = {19-02 (16E40 19D55 55-02 55N99)},
  MRNUMBER = {3013261},
MRREVIEWER = {Charles\ Weibel},
}

@article {HHR,
    AUTHOR = {Hill, M. A. and Hopkins, M. J. and Ravenel, D. C.},
     TITLE = {On the nonexistence of elements of {K}ervaire invariant one},
   JOURNAL = {Ann. of Math. (2)},
  FJOURNAL = {Annals of Mathematics. Second Series},
    VOLUME = {184},
      YEAR = {2016},
    NUMBER = {1},
     PAGES = {1--262},
      ISSN = {0003-486X,1939-8980},
   MRCLASS = {55P91 (55N22 55P42 55Q45 55T15 55U35 57R15)},
  MRNUMBER = {3505179},
MRREVIEWER = {Paul\ G.\ Goerss},
       DOI = {10.4007/annals.2016.184.1.1},
       URL = {https://doi.org/10.4007/annals.2016.184.1.1},
}

@misc{HR26,
      title={Equivariant localizing motives and multiplicative norms on algebraic {K}-theory}, 
      author={Kaif Hilman and Maxime Ramzi},
      year={2026},
      eprint={2603.15873},
      archivePrefix={arXiv},
      primaryClass={math.KT},
      url={https://arxiv.org/abs/2603.15873}, 
}

@article{Hof25,
    title = {Determining unit groups and {$K_1$} of finite rings},
    journal = {Journal of Algebra},
    volume = {693},
    pages = {510-530},
    year = {2026},
    issn = {0021-8693},
    doi = {https://doi.org/10.1016/j.jalgebra.2026.01.019},
    url = {https://www.sciencedirect.com/science/article/pii/S002186932600027X},
    author = {Tommy Hofmann},
}

@article{Nak11a,
title = {Ideals of {T}ambara functors},
journal = {Advances in Mathematics},
volume = {230},
number = {4},
pages = {2295-2331},
year = {2012},
issn = {0001-8708},
author = {Hiroyuki Nakaoka}
}

@article {NS18,
    AUTHOR = {Nikolaus, Thomas and Scholze, Peter},
     TITLE = {On topological cyclic homology},
   JOURNAL = {Acta Math.},
  FJOURNAL = {Acta Mathematica},
    VOLUME = {221},
      YEAR = {2018},
    NUMBER = {2},
     PAGES = {203--409},
      ISSN = {0001-5962,1871-2509},
   MRCLASS = {55U35 (16E40 18E30 19D99)},
  MRNUMBER = {3904731},
MRREVIEWER = {Geoffrey\ M. L. Powell},
       DOI = {10.4310/ACTA.2018.v221.n2.a1},
       URL = {https://doi.org/10.4310/ACTA.2018.v221.n2.a1},
}

@misc{OSCAR,
  key          = {OSCAR},
  organization = {The OSCAR Team},
  title        = {O{SCAR} -- {O}pen {S}ource {C}omputer {A}lgebra {R}esearch system, {V}ersion 1.7.0},
  year         = {2026},
  url          = {https://www.oscar-system.org},
}

@article{Qui72b,
    AUTHOR = {Quillen, Daniel},
     TITLE = {On the cohomology and {$K$}-theory of the general linear groups over a finite field},
   JOURNAL = {Ann. of Math. (2)},
  FJOURNAL = {Annals of Mathematics. Second Series},
    VOLUME = {96},
      YEAR = {1972},
     PAGES = {552--586},
      ISSN = {0003-486X},
}

@article {TW95,
    AUTHOR = {Th\'evenaz, Jacques and Webb, Peter},
     TITLE = {The structure of {M}ackey functors},
   JOURNAL = {Trans. Amer. Math. Soc.},
  FJOURNAL = {Transactions of the American Mathematical Society},
    VOLUME = {347},
      YEAR = {1995},
    NUMBER = {6},
     PAGES = {1865--1961},
      ISSN = {0002-9947,1088-6850},
   MRCLASS = {20C20 (20J05)},
  MRNUMBER = {1261590},
MRREVIEWER = {J.\ L.\ Alperin},
       DOI = {10.2307/2154915},
       URL = {https://doi.org/10.2307/2154915},
}

@book{Wei13,
    AUTHOR = {Weibel, Charles A.},
     TITLE = {The {$K$}-book},
    SERIES = {Graduate Studies in Mathematics},
    VOLUME = {145},
      NOTE = {An introduction to algebraic $K$-theory},
 PUBLISHER = {American Mathematical Society, Providence, RI},
      YEAR = {2013},
     PAGES = {xii+618},
      ISBN = {978-0-8218-9132-2},
   MRCLASS = {19-01},
  MRNUMBER = {3076731},
MRREVIEWER = {L.\ N.\ Vaserstein},
       DOI = {10.1090/gsm/145},
       URL = {https://doi.org/10.1090/gsm/145},
}

@article{Wis24,
title = {A classification of ${C}_{p^n}$-{T}ambara fields},
journal = {Journal of Algebra},
volume = {682},
pages = {233-246},
year = {2025},
issn = {0021-8693},
doi = {https://doi.org/10.1016/j.jalgebra.2025.05.028},
url = {https://www.sciencedirect.com/science/article/pii/S0021869325003333},
author = {Noah Wisdom},
}

@article{Yos83,
  title={On {$G$}-functors ({I}{I}): Hecke operators and {$G$}-functors},
  author={Tomoyuki Yoshida},
  journal={Journal of The Mathematical Society of Japan},
  year={1983},
  volume={35},
  pages={179-190},
}
